\documentclass[3p,11pt]{elsarticle}
\newcommand{\be}{\begin{eqnarray}}
\newcommand{\ee}{\end{eqnarray}}
\newcommand{\bc}{\begin{center}}
\newcommand{\ec}{\end{center}}

\newcommand{\bq}{\begin{que}}
\newcommand{\eq}{\end{que}}
\newcommand{\bqi}{\begin{quei}}
\newcommand{\eqi}{\end{quei}}
\newenvironment{proof}{{\textbf{\textit{Proof.}}}}{\linespread{1.66}\hspace{\stretch{1}}{$\square$}}

\usepackage{pifont}
\usepackage{amsfonts}
\usepackage{amsmath}
\usepackage{amssymb}
\usepackage{latexsym}
\usepackage{graphicx}
\usepackage{epsf}
\usepackage{epstopdf}
\usepackage{latexsym}
  

\newtheorem{teor}{Theorem}[section]

\newtheorem{lemma}{Lemma}[section]

\newtheorem{Rem}{Remark}[section]

  

\begin{document}


\begin{frontmatter}

\title{Modeling transient currents in time-of-flight experiments with tempered time-fractional diffusion equations
}

\author[MLM]{Maria Lu\'isa Morgado  \corref{cor1} 
} \ead{luisam@utad.pt}

\author[LFM]{Lu\'is Filipe Morgado 
} \ead{lmorgado@utad.pt}

\address[MLM]{Center for Computational and Sthocastic Mathematics, Instituto Superior T\'ecnico and Department of Mathematics, University of Tr\'as-os-Montes e Alto Douro, UTAD, Quinta de Prados 5001-801, Vila Real, Portugal}
\address[LFM]{Instituto de Telecomunica\c c\~oes, Lisboa and Department of Physics, University of Tr\'as-os-Montes e Alto Douro, UTAD, Quinta de Prados 5001-801, Vila Real, Portugal}

\pagestyle{plain} \setcounter{page}{1}

\begin{abstract}
In this work we use tempered fractional advection-diffusion equations to model the dispersive transport in disordered materials. A numerical method is derived to approximate the solution of such differential models and we prove that it is convergent and stable. Two numerical examples are presented. The first one,  with known analytical solution, is given to illustrate the performance of the numerical scheme; the second one is used to show that such models are appropriate to model time of flight transient currents  for some disordered materials.
\end{abstract}


\end{frontmatter}

\section{Introduction}
\
\noindent

Nowadays, there is no doubt that Fractional Calculus, which consists of the study of integral and differential operators of non-integer order, is a powerful tool in the modeling of many processes in science and engineering (see for example the books \cite{Baleanu_book}, \cite{DiethelmBook}, \cite{SKM} and the references therein). Although it is more than three hundred years old, only in the last decades it has been investigated intensively. Since its very beginning, that many recognized scientists have contributed to their development and along these last three centuries several definitions (not always equivalent) of fractional derivatives, generalising the integer-order ones, arose. We believe that the Riemann-Liouville and the Caputo derivatives are the most commonly found in the literature, mainly the last ones when application problems are considered (\cite{DiethelmBook}).

Due to the features of fractional integral and differential operators, the need for numerical methods is even more evident in the fractional setting and their development poses us new challenges that did not occur when integer-order models were considered. These challenges may be resumed in two items: the computational effort due to the nonlocality of the fractional operators and the accuracy of the numerical schemes due to the singular nature of the solutions of fractional differential equations. Although numerous papers have been reported in the last decades concerning the numerical approximation of fractional differential equations, we think that in many of them, these two issues were not fully addressed,  and hence further investigation is needed within this respect.

Here we are concerned with the modeling of the the transient current in time-of-flight (ToF) experiments in disordered materials. Such experiments consist on the measurement of the transient current generated by the movement, due to an externally applied electric field $E$, of the excess charge carriers  usually generated by a laser  pulse. Several results from this experiment for disordered materials, namely organic semiconductors, usually exhibit an anomalous dispersive behavior (\cite{Scher1975}), when the transient current $I(t)$ curve presents two regions with power-law behavior, separated by the  "transit time" $t_{tr}$:
\begin{equation}\label{It}
I(t)\sim\begin{cases}
t^{-1+\alpha} &\text{if $t<t_{tr}$}\\
t^{-1-\alpha} &\text{if $t>t_{tr}$}
\end{cases}, 0<\alpha< 1.
\end{equation}
From experimental $I(t)$ curves, usually the $t_{tr}$ is obtained, graphically, from the intersection of the two power-law curves, which is then used to determine the carrier drift mobility.
This behavior was attributed to the trapping of carriers, in localized states distributed in the mobility gap, for times $\tau$, determined by a relaxation function with an asymptotic time dependence of the form $\sim \tau^{-\alpha}$, with non integer dispersion parameter $\alpha$. Alternative physical explanations involve other mechanisms such as phonon assisted hopping conduction and percolation through conducting states.
In some cases, as in \cite{Takeda}, the transient current reveals two different values for the dispersive parameter, $\alpha_i$ for shorter times and $\alpha_f$ for longer times. One possible justification for this fact is based on the assumption that the energy distribution of deep traps is in fact truncated leading to a truncated waiting time distribution (\cite{UchaikinBook}).

This paper continues the investigation initiated in \cite{Morgado}, where the following class of initial-boundary value problem was considered:
\begin{equation}\label{TFADE}
D^{\alpha}_tu(x,t)=-v\frac{\partial u(x,t)}{\partial x}+D\frac{\partial ^{2}u(x,t)}{\partial x^{2}}+f(x,t),~t \in (0,T],~x \in (0,L),
\end{equation}
with initial condition
\begin{equation}\label{init_cond}
u(x,0)=g(x), \quad x\in (0,L),
\end{equation}
and boundary conditions
\begin{equation}\label{bound_cond}
u(0,t)=\phi_0(t), ~~u(L,t)=\phi_L(t), \quad t\in(0,T],
\end{equation}
where $0< \alpha <1$ and the fractional derivative is of the Caputo type which, for the considered values of $\alpha$, is given by (\cite{DiethelmBook}):
\begin{equation}\label{CapDer}
D_t^\alpha y(t) =\frac{1}{\Gamma(1-\alpha)}\int_0^t (t-s)^{-\alpha}y'(s)~ds.
\end{equation}
The unknown function $u$ is usually referred as the solution concentration, $v>0$ is the average fluid velocity, and $D>0$ is the diffusion coefficient. We assume that $g$, $f$, $\phi_0$ and $\phi_L$ are continuous functions in their respective domains.\\
In \cite{Morgado}, an implicit numerical scheme on a time graded mesh was developed to approximate the solution of the initial-boundary value problem above. The reason for using a graded mesh in the time variable, is due to the fact that, as it is known (\cite{Philippa2011}), the solution of this type of problems exhibits a sharp behavior near the origin (which is in agreement with the fact that, usually, fractional differential equations have singularities at $t=0$, in the sense that the derivatives of the solution may become unbounded near that point), and therefore in order to obtain reasonable accurate results it is convenient to use small step-sizes near that point. This was achieved by considering a partition of the time interval into $n$ subintervals defined by the mesh-points:
\begin{equation}\label{grad}t_i=\left( \frac i n\right)^rT,\end{equation}
where $r \ge 1$ is the so-called grading exponent. The length of each one of the intervals defined with this partition is:
\[\tau_i=t_{i+1}-t_i=\frac{(i+1)^r-i^r}{n^r}T, \quad i=0,1, \ldots,n-1.\]
Note that if $r=1$ we obtain a uniform mesh, that is, a mesh where all the subintervals in the partition have the same length ($\tau_i=\tau=\frac T n, ~~i=0,1, \ldots,n-1$), while if $r>1$, the grid-points are more densely placed in the left-hand side of the interval $[0,T]$, as desired.\\
In \cite{Morgado}, the proper choice of the grading exponent was left for future work, and this is one of the aspects we want to address here, basing us on the results of \cite{Stynes}. Moreover, the scheme presented in \cite{Morgado} was first order convergent in space, while here we develop a second order convergent method with respect to the space variable. Another contribution here is that we will generalize the model (\ref{TFADE}) by considering tempered Caputo derivatives instead of the Caputo ones.

By considering the probability density function of L\'evy distributions, the usual (non-tempered) fractional Riemann-Liouville derivatives of order $\alpha$ arise (and consequently, the Caputo derivatives), but since the probability density function of L\'evy distributions decays as $\displaystyle \left|t\right|^{-1-\alpha}$, first or second moments will be divergent (see \cite{Uchaikin1999}). By exponentially tempering the probability of large jumps of L\'evy flights with a parameter $\lambda>0$, the density function decays as  $\displaystyle \left|t\right|^{-1-\alpha}e^{-\lambda \left|t\right|}$ and then finite moments are obtained and the definition of tempered fractional derivatives arise (\cite{Baumer}):
\begin{eqnarray}
\label{LeftC}
\mathbb{D}_t^{\alpha,\lambda}\left(y(t)\right)&=&e^{-\lambda t}D_t^{\alpha}\left(e^{\lambda t} y(t)\right)\\
\nonumber &=&\frac{e^{-\lambda t} }{\Gamma(1-\alpha)}\int_{0}^{t}\frac{1}{(t-s)^{\alpha}}\frac{d\left(e^{\lambda s}y(s)\right)}{d s} ds,~~0<\alpha<1,~\lambda\ge 0.
\end{eqnarray}
Note that if in the equation above, we consider $\lambda=0$, the definition of the usual Caputo derivative (\ref{CapDer}) is recovered.\\
Hence, here we will consider the following model (which obviously reduces to (\ref{TFADE})-(\ref{bound_cond}) in the case where we have $\lambda=0$ and absorbing boundary conditions):
\begin{eqnarray}&&\label{TTFADE}\mathbb{D}_t^{\alpha,\lambda}u(x,t)=-v\frac{\partial u(x,t)}{\partial x}+D\frac{\partial ^{2}u(x,t)}{\partial x^{2}}+f(x,t),~t \in (0,T],~x \in (0,L),\\
&&\label{Tinit_cond}u(x,0)=g(x), \quad x\in (0,L),\\
&&\label{Tbound_cond}u(0,t)=0, ~~u(L,t)=0, \quad t\in(0,T],
\end{eqnarray}
where $\displaystyle \mathbb{D}_t^{\alpha,\lambda}u(x,t)$ is the tempered Caputo derivative with respect to the variable $t$ of the function $u(x,t)$. We also assume $f(x,t)\equiv 0$, but we choose to keep this term in the model above for technical reasons. In fact, it will give us the opportunity to build problems with known analytical solutions and, in this way, we will be able to compare our numerical results with the analytical ones.

The paper is organised in the following way: in the forthcoming section we describe the numerical method to solve (\ref{TTFADE})-(\ref{Tbound_cond}) and we prove that it is convergent and stable. In section \ref{numerics} we firt test the robustness of the numerical method through some numerical experiments for an example with known analytical solution and then we use it to model the transient current in ToF experiments. We end with some conclusions about the advantages of considering this type of equations to model such processes.

\section{Numerical scheme}
In this section we develop an implicit numerical method for  the approximate solution of (\ref{TTFADE})-(\ref{Tbound_cond}). First, taking (\ref{LeftC}) into account, we note that (\ref{TTFADE}) can be written as
\[D_t^{\alpha}\left(e^{\lambda t}u(x,t)\right)=-v\frac{\partial \left(e^{\lambda t}u(x,t)\right)}{\partial x}+D\frac{\partial ^{2}\left(e^{\lambda t}u(x,t)\right)}{\partial x^{2}}+e^{\lambda t}f(x,t).\] 
If we consider the function
\begin{equation}
\label{relat}y(x,t)=e^{\lambda t}u(x,t),
\end{equation}
and we determine the solution $y(x,t)$ of problem:
\begin{eqnarray}\label{Eqaux}
&&D_t^{\alpha}y(x,t)=-v\frac{\partial y(x,t)}{\partial x}+D\frac{\partial ^{2}y(x,t)}{\partial x^{2}}+e^{\lambda t}f(x,t),~t \in (0,T],~x \in (0,L),\\
&&\label{IC}y(x,0)=g(x), \quad x\in (0,L),\\
&&\label{BC}
y(0,t)=0, ~~y(L,t)=0, \quad t\in(0,T],
\end{eqnarray}
then the solution of (\ref{TTFADE})-(\ref{Tbound_cond}) is obtained through
\[u(x,t)=e^{-\lambda t}y(x,t).\]
Therefore, in what follows we present a numerical scheme for the solution of (\ref{Eqaux})-(\ref{BC}).
In order to do that, we consider a uniform mesh in the interval $[0,L]$, defined by the grid-points $x_i=ih$, $i=0,1,\ldots,K$, where $h=\frac L K$, and we use the following second order finite difference approximations, assuming that the solution possess fourth order continuous derivatives with respect to $x$:
\begin{eqnarray}
\label{AppSpDer1} \frac{\partial y(x_i,t)}{\partial x}&\approx &\frac{y(x_{i+1},t)-y(x_{i-1},t)}{2h},\\
\label{AppSpDer2} \frac{\partial ^{2}y(x_i,t)}{\partial x^{2}}&\approx & \frac{y(x_{i+1},t)-2y(x_{i},t)+y(x_{i-1},t)}{h^2}, \quad i=1,\ldots,K-1.
\end{eqnarray}

For the numerical approximation of the Caputo derivative of order $\alpha$  on the interval $[0,T]$, we will use the non-uniform mesh (\ref{grad}),
and use the following approximation for the Caputo derivative (see \cite{Morgado}):
\begin{equation}\label{CapDerGradMesh1}
D^\alpha y(t_k) \approx \frac{1}{\Gamma(2-\alpha)}\sum_{j=0}^{k-1}\tau_j^{-\alpha}a_{j,k}\left(y(t_{j+1})-y(t_j)\right)=\tilde{D}^{\alpha} y_k,
\end{equation}
where
\begin{equation}\label{coef}
a_{j,k}=\left( \frac{k^r-j^r}{(j+1)^r-j^r}\right)^{1-\alpha}-\left(\frac{k^r-(j+1)^r}{(j+1)^r-j^r}\right)^{1-\alpha}, \quad j=0,1, \ldots, k-1,~~k=1, \ldots,n.
\end{equation}
The analytical solution of this kind of problems with $f(x,t)\equiv 0$ is known  (see, for example Eq. (4) in \cite{Philippa2011}) showing that a singular behavior of the solution near the origin in time can be expected, and therefore, it is reasonable to consider that
\begin{equation}\label{sing_behav}
\left|\frac{\partial ^{\ell}}{\partial t^{\ell}}y(x,t)\right| \le C\left(1+t^{\alpha -\ell}\right), \quad \ell=0,1,2,
\end{equation}
for some positive constant $c$ and for all $(x,t)\in [0,L] \times (0,T]$.
Hence, under this assumption, and concerning the order of the approximation we have (see \cite{Stynes}):
\begin{eqnarray}
\nonumber \left|D_t^\alpha y(t_k) -\tilde{D}^{\alpha} y_k\right|&\le &Ck^{-\min\left\{2-\alpha,r \alpha\right\}}.
\end{eqnarray}

Using (\ref{CapDerGradMesh1}), we obtain:
\begin{equation}\label{AppTimeDer}
D_t^{\alpha}y(x_i,t_l)\approx \frac{1}{\Gamma(2-\alpha)}\sum_{j=0}^{l-1}\tau_{j}^{-\alpha}a_{j,l}\left(y(x_i,t_{j+1})-y(x_i,t_{j})\right), \quad i=1,\ldots,K-1,~~l=1,\ldots, n,
\end{equation}
where the coefficients $a_{j,l}$ are defined in (\ref{coef}). Denoting by $Y_i^l\approx y(x_i,t_l)$, $f_i^l=f(x_i,t_l)$ and substituting (\ref{AppSpDer1}), (\ref{AppSpDer2}) and (\ref{AppTimeDer}) in (\ref{Eqaux}), we obtain the following implicit numerical scheme:
\begin{eqnarray}\label{NumMeth} &&\frac{1}{\Gamma(2-\alpha)}\sum_{j=0}^{l-1} \tau_{j}^{-\alpha}a_{j,l}\left(Y_i^{j+1}-Y_i^{j}\right)=-v \frac{Y_{i+1}^l-Y_{i-1}^l}{2h}+D\frac{Y_{i+1}^l-2Y_{i}^l+Y_{i-1}^l}{h^2}+e^{\lambda t_l}f_i^l, \\
\nonumber && i=1,\ldots,K-1,~l=1,\ldots,n,
\end{eqnarray}
where, according to the initial and boundary conditions (\ref{IC}) and (\ref{BC}), we have
\begin{eqnarray}
\nonumber Y_i^0&=&g(x_i),~~i=1,\ldots,K-1,\\
\nonumber Y_0^l&=&0,~~Y_K^l=0,~~l=1,\ldots,n.
\end{eqnarray}
After having determined the unknowns $Y_i^l$, $i=1,\ldots,K-1,~l=1,\ldots,n$, the solution of (\ref{TTFADE})-(\ref{Tbound_cond}) at the mesh-points will be given by:
\[u(x_i,t_l)\approx U_i^l=e^{-\lambda t_l} Y_i^l.\]
\subsection{Stability of the numerical scheme}\label{StabCon}
In this subsection, we prove the stability of the numerical scheme just described, which can be written as
\begin{equation}\label{meth}
T_1Y_i^l=T_2Y_i^{l-1}+e^{\lambda t_l}f_i^l, ~~i=1,\ldots,K-1,~l=1,\ldots,n,
\end{equation}
where 
\begin{eqnarray}
\nonumber T_1Y_i^l&=&\frac{\tau_{l-1}^{-\alpha}}{\Gamma(2-\alpha)}Y_i^l+v \frac{Y_{i+1}^l-Y_{i-1}^l}{2h}-D\frac{Y_{i+1}^l-2Y_{i}^l+Y_{i-1}^l}{h^2},\\
\nonumber T_2Y_i^{l-1}&=&\frac{\tau_{l-1}^{-\alpha}}{\Gamma(2-\alpha)}Y_i^{l-1}-\frac{1}{\Gamma(2-\alpha)}\sum_{j=0}^{l-2} \tau_{j}^{-\alpha}a_{j,l}\left(Y_i^{j+1}-Y_i^{j}\right).
\end{eqnarray}
We start with some auxiliary results, proved in \cite{Morgado}, that will be needed later.
\begin{lemma}\label{le}
The coefficients $a_{j,l}$, $j=0,\ldots,l-2,~~l=1,\ldots,n$, defined by (\ref{coef}) satisfy:
\begin{eqnarray}
\label{pos} && a_{j,l} > 0,\\
\label{pos1} && \sum_{j=0}^{l-2}(\tau_{j+1}^{-\alpha}a_{j+1,l}-\tau_{j}^{-\alpha}a_{j,l})=-\tau_{0}^{-\alpha}a_{0,l}+\tau_{l-1}^{-\alpha},\\
\label{pos2} && \tau_{j+1}^{-\alpha}a_{j+1,l}>\tau_{j}^{-\alpha}a_{j,l}.
\end{eqnarray}
\end{lemma}
In order to prove the stability of the numerical scheme, let us assume that the initial data has error $\epsilon_i^0$,  that is, let us assume that $\tilde{g}(x_i)=g(x_i)+\epsilon_i^0$, $i=1,2,\ldots,K-1$, and let $Y_i^l$ and $\tilde{Y}_i^l$ be the solutions of (\ref{meth}) corresponding to the initial data $g$ and $\tilde{g}$, respectively. Defining $\epsilon_i^l=Y_i^l-\tilde{Y}_i^l$, we have
\[T_1 \epsilon_i^l=T_2\epsilon_i^{l-1},~~i=1,2,\ldots,K-1, ~l=1,2,\ldots,n.\]
Setting $\left\|E^l\right\|_{\infty}=\max_{1\le i \le K-1}\left|\epsilon_i^l\right|$, we next prove, using mathematical induction, that
\[\left\|E^l\right\|_{\infty} \le \left\|E^0\right\|_{\infty}\]
is satisfied for all $l=1,2,\ldots,n$.\\
For $l=1$, let $p \in \mathbb{N}$ be such that $\left\|E^
1\right\|_{\infty}=\max_{1\le i \le K-1}\left|\epsilon_i^1\right|=\left|\epsilon_p^1\right|$. Then,
\begin{eqnarray}
\nonumber \frac{\tau_0^{-\alpha}}{\Gamma(2-\alpha)}\left\|E^
1\right\|_{\infty}&=&\frac{\tau_0^{-\alpha}}{\Gamma(2-\alpha)}\left|\epsilon_p^1\right|=\frac{\tau_0^{-\alpha}}{\Gamma(2-\alpha)}\left|\epsilon_p^1\right|+\frac{2D}{h^2}\left|\epsilon_p^1\right|-\frac{2D}{h^2}\left|\epsilon_p^1\right|+\frac{2D}{h^2}\left|\epsilon_p^1\right|\\
\nonumber &=& \frac{\tau_0^{-\alpha}}{\Gamma(2-\alpha)}\left|\epsilon_p^1\right|+\frac{2D}{h^2}\left|\epsilon_p^1\right|-\left(\frac{D}{h^2}-\frac{v}{2h}\right) \left|\epsilon_p^1\right|-\left(\frac{D}{h^2}+\frac{v}{2h}\right) \left|\epsilon_p^1\right|
\end{eqnarray}
If $\displaystyle h<\frac{2D}{v}$, then $\displaystyle \frac{2D-vh}{2h^2}>0$, and therefore,
\begin{eqnarray}
\nonumber \frac{\tau_0^{-\alpha}}{\Gamma(2-\alpha)}\left\|E^
1\right\|_{\infty} &\le & \frac{\tau_0^{-\alpha}}{\Gamma(2-\alpha)}\left|\epsilon_p^1\right|+\frac{2D}{h^2}\left|\epsilon_p^1\right|-\left(\frac{D}{h^2}-\frac{v}{2h}\right) \left|\epsilon_{p+1}^1\right|-\left(\frac{D}{h^2}+\frac{v}{2h}\right) \left|\epsilon_{p-1}^1\right|\\
\nonumber & \le & \left|\frac{\tau_0^{-\alpha}}{\Gamma(2-\alpha)}\epsilon_p^1 +v\frac{\epsilon_{p+1}^1-\epsilon_{p-1}^1}{2h}-D\frac{\epsilon_{p+1}^1-2\epsilon_p^1+\epsilon_{p-1}^1}{h^2}\right|\\
\nonumber &=& \left|T_1\epsilon_p^1\right|=\left|T_2\epsilon_p^0\right|=\left|\frac{\tau_0^{-\alpha}}{\Gamma(2-\alpha)}\epsilon_p^0\right|\le \frac{\tau_0^{-\alpha}}{\Gamma(2-\alpha)}\left\|E^
0\right\|_{\infty},
\end{eqnarray}
and then if follows that $\left\|E^
1\right\|_{\infty} \le \left\|E^
0\right\|_{\infty}$.\\
Let us now assume that $\left\|E^
j\right\|_{\infty} \le \left\|E^
0\right\|_{\infty}$, $j=1,\ldots,l-1$, and assume also that $p\in\mathbb{N}$ is such that $\left\|E^
l\right\|_{\infty}=\left|\epsilon_p^l\right|$. Hence following the same steps as above,
\begin{eqnarray}
\nonumber \frac{\tau_{l-1}^{-\alpha}}{\Gamma(2-\alpha)}\left\|E^
l\right\|_{\infty}&=&\frac{\tau_{l-1}^{-\alpha}}{\Gamma(2-\alpha)}\left|\epsilon_p^l\right| \le  \left|T_1\epsilon_p^l\right|=\left|T_2\epsilon_p^{l-1}\right|\\
\nonumber & =& \left| \frac{\tau_{l-1}^{-\alpha}}{\Gamma(2-\alpha)}\epsilon_p^{l-1}-\frac{1}{\Gamma(2-\alpha)}\sum_{j=0}^{l-2} \tau_{j}^{-\alpha}a_{j,l}\left(\epsilon_p^{j+1}-\epsilon_p^{j}\right)\right|\\
\nonumber &=& \frac{1}{\Gamma(2-\alpha)}\left|\tau_0^{-\alpha}a_{0,l}\epsilon_p^0+\sum_{j=0}^{l-2}\left( \tau_{j+1}^{-\alpha}a_{j+1,l}-\tau_j^{-\alpha}a_{j,l}\right) \epsilon_p^{j+1}\right|.
\end{eqnarray}
Using (\ref{pos2}), the induction hypothesis and (\ref{pos1}), it follows that:
\begin{eqnarray}
\nonumber \frac{\tau_{l-1}^{-\alpha}}{\Gamma(2-\alpha)}\left\|E^
l\right\|_{\infty}& \le & \frac{1}{\Gamma(2-\alpha)}\left(\tau_0^{-\alpha}a_{0,l}\left|\epsilon_p^0\right|+\sum_{j=0}^{l-2}\left( \tau_{j+1}^{-\alpha}a_{j+1,l}-\tau_j^{-\alpha}a_{j,l}\right) \left|\epsilon_p^{j+1}\right|\right)\\
\nonumber & \le & \frac{\left\|E^
0\right\|_{\infty}}{\Gamma(2-\alpha)}\left( \tau_0^{-\alpha}a_{0,l}+\sum_{j=0}^{l-2}\left( \tau_{j+1}^{-\alpha}a_{j+1,l}-\tau_j^{-\alpha}a_{j,l}\right)\right)\\
\nonumber &=& \frac{\left\|E^
0\right\|_{\infty}}{\Gamma(2-\alpha)}\tau_{l-1}^{\alpha}.
\end{eqnarray}
We can conclude that $\left\|E^
l\right\|_{\infty} \le \left\|E^
0\right\|_{\infty}$, $l=1,2,\ldots,n$, and then the following result is proved.
\begin{teor}
If $\displaystyle h<\frac{2D}{v}$, the numerical scheme (\ref{meth}) is stable.
\end{teor}
\subsection{Convergence of the numerical scheme}
In order to prove the convergence order of the numerical scheme, let us first note that taking into account (\ref{AppSpDer1}), (\ref{AppSpDer2}),  (\ref{AppTimeDer}) and Theorem 5.2 of \cite{Stynes}, the solution of (\ref{Eqaux})-(\ref{BC}) satisfies:
\begin{eqnarray}\nonumber
&& \frac{1}{\Gamma(2-\alpha)}\sum_{j=0}^{l-1} \tau_{j}^{-\alpha}a_{j,l}\left(y(x_i,t_{j+1})-y(x_i,t_j)\right)=-v \frac{y(x_{i+1},t_l)-y(x_{i-1},t_l)}{2h}\\
\nonumber && +D\frac{y(x_{i+1},t_l)-2y(x_{i},t_l)+y(x_{i-1},t_l)}{h^2}+e^{\lambda t_l}f_i^l+R_i^{l}, ~~i=1,\ldots,K-1,~l=1,\ldots,n,
\end{eqnarray}
where $\left\|R^
{l}\right\|_{\infty} =\max_{1 \le i \le K-1}\left|R^
{l}\right|\le C_1(n^{\beta}+h^2)$, being $\beta=-\min\left\{2-\alpha,r \alpha\right\}$ and $C_1$ a positive constant not depending on $n$ or $h$.\
Define the error at every point of the mesh by
\[e_i^l=y(x_i,t_l)-Y_i^l,~~i=1,\ldots,K-1,~l=1,\ldots,n,\]
and $\mathbf{e}^l=(e_1^l~~e_2^l~\ldots ~e_{K-1}^l)^T$. Obviously $\mathbf{e}^0=(0~~0~\ldots ~0)^T$, and
\[T_1 e_i^l=T_2 e_{i-1}^l+R_i^l,~~i=\,\ldots,K-1,~~l=1,\ldots,n.\]
We first prove the following result:
\begin{lemma}
There exists a positive constant $C_1$ not depending on $n$ and $h$ such that:
\begin{equation}\label{thm}\left\|\mathbf{e}^l\right\|_{\infty} \le \frac{C_1(n^{\beta}+h^2)}{\frac{1}{\Gamma(2-\alpha)}\left( \tau_{l-1}^{-\alpha}-\sum_{j=0}^{l-2}\left( \tau_{j+1}^{-\alpha}a_{j+1,l}-\tau_j^{-\alpha}a_{j,l}\right)\right)},~~l=1,2,\ldots,n.\end{equation}
\end{lemma}
\begin{proof}
We use mathematical induction to prove (\ref{thm}). Similarly to the proof of stability, for $l=1$, let $p \in \mathbb{N}$ be such that $\left\|\mathbf{e}^
1\right\|_{\infty}=\max_{1\le i \le K-1}\left|e_i^1\right|=\left|e_p^1\right|$. Then,
\begin{eqnarray}
\nonumber \frac{\tau_0^{-\alpha}}{\Gamma(2-\alpha)}\left\|\mathbf{e}^
1\right\|_{\infty}&=&\frac{\tau_0^{-\alpha}}{\Gamma(2-\alpha)}\left|e_p^1\right|\\
\nonumber &=& \left|T_1e^1\right|=\left|T_2e_p^0+R_p^1\right|=\left|R_p^1\right|\le \left\|R^
1\right\|_{\infty} \le C_1(n^{\beta}+h^2),
\end{eqnarray}
and then (\ref{thm}) is satisfied for $l=1$. Assume now that 
\[\left\|\mathbf{e}^
j\right\|_{\infty} \le \frac{C_1(n^{\beta}+h^2)}{\frac{1}{\Gamma(2-\alpha)}\left( \tau_{j-1}^{-\alpha}-\sum_{m=0}^{j-2}\left( \tau_{m+1}^{-\alpha}a_{m+1,l}-\tau_m^{-\alpha}a_{m,l}\right)\right)},~~j=\,2,\ldots,l-1,\]and  that $p\in\mathbb{N}$ is such that $\left\|\mathbf{e}^
l\right\|_{\infty}=\left|e_p^l\right|$. Hence,
\begin{eqnarray}
\nonumber && \frac{\tau_{l-1}^{-\alpha}}{\Gamma(2-\alpha)}\left\|\mathbf{e}^
l\right\|_{\infty}=\frac{\tau_{l-1}^{-\alpha}}{\Gamma(2-\alpha)}\left|e_p^l\right| \le  \left|T_1e_p^l\right|=\left|T_2\epsilon_p^{l-1}+R_p^l\right|\\
\nonumber && = \left| \frac{\tau_{l-1}^{-\alpha}}{\Gamma(2-\alpha)}\epsilon_p^{l-1}-\frac{1}{\Gamma(2-\alpha)}\sum_{j=0}^{l-2} \tau_{j}^{-\alpha}a_{j,l}\left(\epsilon_p^{j+1}-\epsilon_p^{j}\right)+R_p^l\right|\\
\nonumber && = \left| \frac{1}{\Gamma(2-\alpha)}\sum_{j=0}^{l-2}\left( \tau_{j+1}^{-\alpha}a_{j+1,l}-\tau_j^{-\alpha}a_{j,l}\right) e_p^j +R_p^l\right|\\
\nonumber && \le  \frac{1}{\Gamma(2-\alpha)}\sum_{j=0}^{l-2}\left( \tau_{j+1}^{-\alpha}a_{j+1,l}-\tau_j^{-\alpha}a_{j,l}\right) \left\|\mathbf{e}^j \right\|_{\infty} +\left\|R^l\right\|_{\infty}.
\end{eqnarray}
Therefore,
\begin{eqnarray}
\nonumber && \frac{\tau_{l-1}^{-\alpha}}{\Gamma(2-\alpha)}\left\|\mathbf{e}^
l\right\|_{\infty} \le  \frac{1}{\Gamma(2-\alpha)}\sum_{j=0}^{l-2}\left( \tau_{j+1}^{-\alpha}a_{j+1,l}-\tau_j^{-\alpha}a_{j,l}\right)\\ 
\nonumber && \times \frac{C_1(n^{\beta}+h^2)}{\frac{1}{\Gamma(2-\alpha)}
\left( \tau_{j-1}^{-\alpha}-\sum_{m=0}^{j-2}\left( \tau_{m+1}^{-\alpha}a_{m+1,l}-\tau_m^{-\alpha}a_{m,l}\right)\right)} +C_1(n^{\beta}+h^2)\\
\nonumber && \le  \frac{1}{\Gamma(2-\alpha)}\sum_{j=0}^{l-2}\left( \tau_{j+1}^{-\alpha}a_{j+1,l}-\tau_j^{-\alpha}a_{j,l}\right) \frac{C_1(n^{\beta} +h^2)}{\frac{1}{\Gamma(2-\alpha)}\left( \tau_{l-1}^{-\alpha}-\sum_{m=0}^{l-2}\left( \tau_{m+1}^{-\alpha}a_{m+1,l}-\tau_m^{-\alpha}a_{m,l}\right)\right)} +\\ \nonumber && +C_1(n^{\beta}+h^2)\\
\nonumber && = \frac{C_1(n^{\beta} +h^2)}{\frac{1}{\Gamma(2-\alpha)}\left( \tau_{j-1}^{-\alpha}-\sum_{m=0}^{l-2}\left( \tau_{m+1}^{-\alpha}a_{m+1,l}-\tau_m^{-\alpha}a_{m,l}\right)\right)}  \frac{\tau_{l-1}^{-\alpha}}{\Gamma(2-\alpha)}
\end{eqnarray}
and then the result is proved. 
\end{proof}\\
The result about the convergence order is given in the next theorem:
\begin{teor}
There exists a positive constant $C$ not depending on $n$ and $h$, such that
\[\left\|\mathbf{e}^l\right\|_{\infty} \le C(n^{\beta}+h^2),~~l=1,\ldots,n,\]
where $\displaystyle \beta=-\min\left\{2-\alpha,r \alpha\right\}$.
\end{teor}
\begin{proof}
First note that from Lemma \ref{le},
\[\tau_{l-1}^{-\alpha}-\sum_{j=0}^{l-2}\left( \tau_{j+1}^{-\alpha}a_{j+1,l}-\tau_j^{-\alpha}a_{j,l}\right)=\tau_0^{-\alpha}a_{0,l}=\tau_0^{-\alpha}\left( \left( l^r\right)^{1-\alpha}-\left( l^r-1\right)^{1-\alpha}\right).\]
Since
\[\lim\limits_{l\rightarrow \infty} \frac{\left( l^r\right)^{-\alpha}}{\left( l^r\right)^{1-\alpha}-\left( l^r-1\right)^{1-\alpha}}=\lim\limits_{\eta\rightarrow \infty} \frac{\eta^{-\alpha}}{\eta^{1-\alpha}-\left( \eta-1\right)^{1-\alpha}}=\frac{1}{1-\alpha},\]
there must exist a positive constant $C_2$, not depending on $n$ and $h$, such that (\ref{thm}) becomes
\[\left\|\mathbf{e}^l\right\|_{\infty} \le \frac{C_1C_2(\tau +h)}{\frac{1}{\Gamma(2-\alpha)}(l^r)^{-\alpha}\tau_0^{-\alpha}}=\frac{C_1C_2(n^{\beta}+h^2)}{\frac{1}{\Gamma(2-\alpha)}t_l^{-\alpha}} \le C(n^{\beta}+h^2).\] 
\end{proof}
\begin{Rem}
By using this graded mesh, and taking into account this last Theorem, we see that the time accuracy of $\left(2-\alpha\right)$ is recovered in the case of nonsmooth enough solutions, if we choose the grading exponent to be $r=\frac{2-\alpha}{\alpha}$.
\end{Rem}
\section{Numerical results}\label{numerics}
To illustrate the proposed numerical method, we first consider an example with a known analytical solution.  By taking in (\ref{TTFADE})-(\ref{Tbound_cond}) $v=D=1$, $T=L=1$ and $g(x)=0$, the function $\displaystyle f(x,t)$ is defined in such a way that the exact solution is given by $\displaystyle u(x,t)=e^{-\lambda t}t^{\alpha}x(1-x)$. We will only be concerned with the improvement of the time accuracy, which is achieved with the graded time mesh, and then in our numerical experiments, we have fixed a small space stepsize, $h=10^{-4}$.\\
In Table \ref{tab1} we present, for this example with $\alpha=0.5$ and $\lambda=1$, the values of the maximum of the absolute errors at the discretization points:
\[E=\max_{i,j}\left|u(x_i,t_j)-U_i^j\right|,\]
as well as experimental time convergence orders (EOC). As we can see from the results in that table, the  time convergence order is increased by using a graded mesh with grading exponent $r=3$.
\begin{table}[htb]
\centering
\begin{tabular}{c||c}
\hline\hline
$r=1$ & $r=3$\\
\hline \hline
\begin{tabular}{c|cc}
$n$ & $E$ & $EOC$\\
\hline
5 & $3.86 \times 10^{-3}$ & -\\
10 & $4.11 \times 10^{-3}$ & -0.08\\
20 & $3.92 \times 10^{-3}$ & 0.07\\
40 & $3.54 \times 10^{-3}$ &  0.15\\
80 & $3.05 \times 10^{-3}$ &  0.21\\
160 & $2.54 \times 10^{-3}$ & 0.26
\end{tabular} &
\begin{tabular}{cc}
$E$ & $EOC$\\
\hline
$2.72 \times 10^{-3}$ & -\\
$1.48 \times 10^{-3}$ & 0.88\\
$7.57 \times 10^{-4}$ & 0.97\\
$3.38 \times 10^{-4}$ & 1.16\\
$1.39 \times 10^{-4}$ & 1.29\\
$5.40 \times 10^{-5}$ & 1.36
\end{tabular}\\
\hline \hline
\end{tabular}
\caption{Absolute errors and experimental time convergence order for the first example with $\alpha=0.5$ and $\lambda=1$.}\label{tab1}
\end{table}
To illustrate that the method can also be applied for usual (non-tempered) fractional differential equations, we present in Table \ref{tab2}  the obtained numerical values in the case where $\alpha=0.25$ and $\lambda=0$.

\begin{table}[htb]
\centering
\begin{tabular}{c||c}
\hline\hline
$r=1$ & $r=7$\\
\hline \hline
\begin{tabular}{c|cc}
$n$ & $E$ & $EOC$\\
\hline
$5$  & $3.94 \times 10^{-3}$ & -\\
$10$ & $3.90 \times 10^{-3}$ & 0.01\\
$20$ & $3.79 \times 10^{-3}$ & 0.04\\
$40$ & $3.66 \times 10^{-3}$ & 0.05\\
$80$ & $3.52 \times 10^{-3}$ & 0.06\\
$160$& $3.36 \times 10^{-2}$ & 0.07
\end{tabular} &
\begin{tabular}{cc}
$E$ & $EOC$\\
\hline
$4.09 \times 10^{-3}$ & -\\
$2.47 \times 10^{-3}$ & 0.72\\
$1.12 \times 10^{-3}$ & 1.14\\
$4.33 \times 10^{-4}$ & 1.37\\
$1.53 \times 10^{-4}$ & 1.50\\
$5.10 \times 10^{-5}$ & 1.59
\end{tabular}\\
\hline \hline
\end{tabular}
\caption{Absolute errors and experimental time convergence order for the first example with $\alpha=0.25$ and $\lambda=0$.}\label{tab2}
\end{table}

Next we consider the problem we address here, the modeling of transient currents in ToF experiments.
The total measured current $I(t)$, produced by the extraction of carriers from the space between the electrodes, placed at $x=0$ and $x=L$, is given (\cite{Philippa2011}) by the space average of the current density $j(x,t)$
\begin{equation}\label{curr}
I(t)=\frac{1}{L}\int^{L}_{0}j(x',t)dx',
\end{equation}
and since
\begin{equation}\label{curra}
j(x',t)=-\frac{d}{dt}\int^{x'}_{0}qu(x,t)dx,
\end{equation}
where $q$ is the carrier electrical charge, we get
\begin{equation}\label{currf}
\frac{I(t)}{q}=-\frac{d}{dt}\int^{L}_{0}(L-x)u(x,t)dx.
\end{equation}

We use a set of data from \cite{Takeda} corresponding to measured currents in amorphous boron ($\beta$-rhombohedral boron), presented in Figure \ref{fig:fig1} (symbols). As mentioned above, this data do not allow a description based on a single dispersive parameter. Considering the model with tempered fractional derivatives we obtain the approximate curve (solid line) that fits quite well to the data with $\alpha=0.66$ and $\lambda=1.0 t_T^{-1}$. In the simulation we consider  that the initial carrier number density is gaussian distributed around $x=0.2L$, namely $g(x)\propto\exp(-2\times10^3(x-0.2L)^2)$, with $v=0.38/L(t_T^{-\alpha})$, $D=2.7\times10^{-3}/L^2(t_T^{-\alpha})$ and $r=3$.\\
It should be mentioned that a good fit could not have been obtained with model (\ref{TFADE})-(\ref{bound_cond}) with a single value for the dispersive parameter $\alpha$, as explained in the Introduction. Hence, the model we propose here is more flexible in the sense that for materials that could be accurately described with the model in \cite{Morgado}, we conclude that they continue to be reproduced by using model (\ref{TTFADE})-(\ref{Tbound_cond}) with $\lambda=0$. For those material, as the $\beta$-rhombohedral boron, we need to use model (\ref{TTFADE})-(\ref{Tbound_cond}) and adjust not only the dispersive parameter $\alpha$, but also the temepered parameter $\lambda$.

\begin{figure}\label{fig:fig1}
\centering
\resizebox{0.5\hsize}{!}{\includegraphics{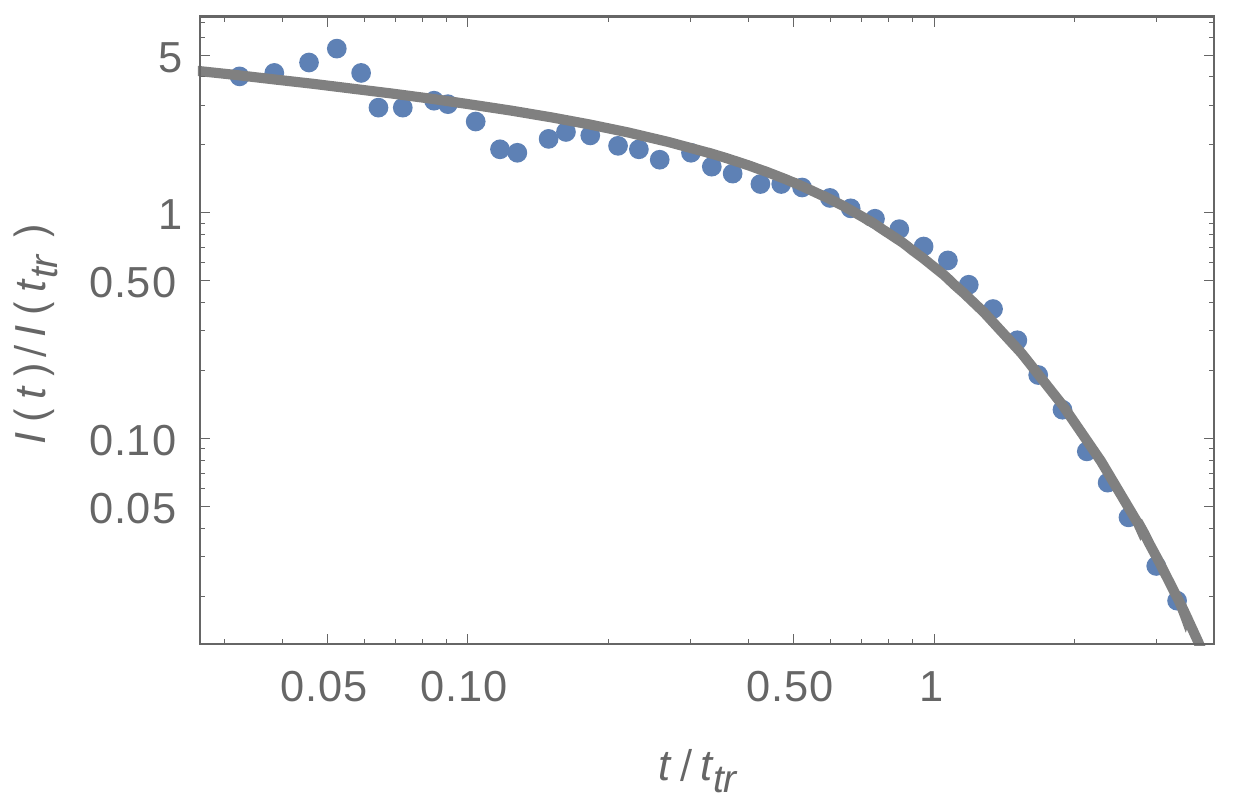}}
\caption{Transient current for amorphous boron ($\beta$-rhombohedral boron) from \cite{Takeda}. Dots: experimental data; Solid line: fitted curve from tempered fractional derivatives model with  $\alpha=0.66$ and $\lambda=1.0 t_T^{-1}$.}
\end{figure}
\section{Conclusions}
Here we have used tempered fractional diffusion-advection equations to model transient currents in ToF experiments in disordered materials. A numerical scheme has been developed, which is stable and $O(n^{-2-\alpha}+h^2)$ accurate. The numerical results show that for some materials it is not enough to consider fractional differential models as (\ref{TFADE}), and by using tempered fractional derivatives instead of the usual ones, the models can be in fact improved.

\vspace{0.2cm}

{\bf Acknowledgments} First author is supported by Portuguese funds through CEMAT - Center for Computacional and Sthocastic Mathematics, and the Portuguese Foundation for Science and Technology (FCT-Funda\c{c}\~ao para a Ci\^encia e a Tecnologia), within project  UID/Multi/04621/2013. Second author acknowledges financial support from Portuguese Foundation for Science and Technology (FCT),
under the contract PEst-OE/EEI/LA0008/2013.


\begin{thebibliography}{99}

\bibitem{Baleanu_book}D. Baleanu, K. Diethelm, E. Scalas and J. J. Trujillo, Fractional calculus models and numerical methods. Series on Complexity Nonlinearity and Chaos, World Scientific, Boston (2012).

\bibitem{Baumer}B.~Baeumer and  M.~M.~Meerschaert, Tempered stable L\'evy motion and transient super-diffusion. J.
Comput. Appl. Math. 233 (10), (2010) 2438--2448. 

\bibitem{DiethelmBook}Kai Diethelm, The analysis of fractional differential
equations: An application-oriented exposition using differential operators of
Caputo type. Springer (2010).


\bibitem{Morgado}L.~F.~Morgado and M.~L.~Morgado, Numerical modelling transient current in the time-of-flight experiment with time-fractional advection-diffusion equations, Journal of Mathematical Chemistry 53, 3 (2015) 958--973.


\bibitem{Philippa2011} B. W. Philippa,  R. D. White and  R. E. Robson,  Analytic solution of the fractional advection-diffusion equation for the time-of-flight experiment in a finite geometry,  Phys. Rev. E {\bf 84} (2011), 041138.

\bibitem{SKM}
S. G. Samko, A. A. Kilbas and O. I. Marichev, Fractional Integrals and Derivatives: Theory and Applications.
Gordon and Breach, Yverdon (1993).

\bibitem{Scher1975}
Scher, H. and Montroll, E., Anomalous transit-time dispersion in amorphous solids,  Phys. Rev. B {\bf 12} (1975) 2455--2477.

\bibitem{Stynes}
Martin Stynes, Eugene O'Riordan, and Jos\'{e} Luis Gracia,  Error analysis of a finite difference method on graded meshes for a time-fractional
diffusion equation, SIAM J. Numer. Anal., 55(2) (2017) 1057--1079.  

\bibitem{Takeda}Takeda M., Kimura K., Murayama K. Transient Photocurrent Studies on Amorphous and  $\beta$-Rhombohedral Boron, J. of Solid State Chemistry 133 (1997) 201.

\bibitem{Uchaikin1999}Uchaikin, V.~V. and Zolotarev, V.~M., Chance and stability. Stable distributions and their applications, VSP, Utrecht, The Netherlands, 1999.

\bibitem{UchaikinBook}V.~V.~Uchaikin and R.~T.~Sibatov, Fractional Kinetics in Solids: Anomalous Charge Transport in Semiconductors, Dielectrics and Nanosystems, World Scientific, 2012


\end{thebibliography}
\end{document}